\documentclass[12pt]{amsart}
\usepackage[margin=1in]{geometry} 
\usepackage{amsmath, amsthm, amsfonts, amssymb, mathrsfs, stmaryrd}
\usepackage[usenames, dvipsnames]{color}
\usepackage[bookmarks, bookmarksdepth=2, colorlinks=true, linkcolor=blue, citecolor=blue, urlcolor=blue]{hyperref}
\usepackage{mathtools}
\usepackage{tikz-cd}
\usepackage{enumitem}

\newcommand{\cA}{\mathcal{A}}

\newcommand{\cC}{\mathcal{C}}
\newcommand{\cB}{\mathcal{VI}}
\newcommand{\VBm}{\mathcal{VB}}

\newcommand{\cD}{\mathcal{D}}

\newcommand{\cE}{\mathcal{E}}

\newcommand{\cH}{\mathcal{H}}

\newcommand{\cI}{\mathcal{I}}

\newcommand{\cL}{\mathcal{L}}

\newcommand{\cM}{\mathcal{L}^{\gen}}

\newcommand{\bN}{\mathbf{N}}

\newcommand{\bQ}{\mathbf{Q}}

\newcommand{\rR}{\mathrm{R}}

\newcommand{\fS}{\mathfrak{S}}

\newcommand{\bV}{\mathbf{V}}

\newcommand{\bZ}{\mathbf{Z}}

\newcommand{\fa}{\mathfrak{a}}

\newcommand{\fb}{\mathfrak{b}}

\newcommand{\fm}{\mathfrak{m}}

\newcommand{\Sm}{A/\fm^{[q]}}

\renewcommand{\phi}{\varphi}
\renewcommand{\emptyset}{\varnothing}

\newcommand{\lw}{{\textstyle \bigwedge}}


\newcommand{\VI}{\operatorname{VI}}
\newcommand{\VB}{\operatorname{VB}}
\newcommand{\VIg}{\cB^\gen} 
\DeclareMathOperator{\im}{im} 
\DeclareMathOperator{\coker}{coker}

\DeclareMathOperator{\ext}{Ext}

\DeclareMathOperator{\End}{End}

\DeclareMathOperator{\Sym}{Sym}

\DeclareMathOperator{\Mod}{Mod}

\DeclareMathOperator{\Hom}{Hom}
\DeclareMathOperator{\Rep}{Rep}

\newcommand{\id}{\mathrm{id}}

\newcommand{\pol}{\mathrm{pol}}

\newcommand{\tors}{\mathrm{tors}}

\newcommand{\GL}{\mathbf{GL}}

\makeatletter
\@addtoreset{equation}{section}
\makeatother

\numberwithin{equation}{section}
\newtheorem{theorem}[equation]{Theorem}
\newtheorem{proposition}[equation]{Proposition}
\newtheorem{lemma}[equation]{Lemma}
\newtheorem{corollary}[equation]{Corollary}

\theoremstyle{definition}
\newtheorem{rmk}[equation]{Remark}
\newenvironment{remark}[1][]{\begin{rmk}[#1] \pushQED{\qed}}{\popQED \end{rmk}}
\newtheorem{eg}[equation]{Example}
\newenvironment{example}[1][]{\begin{eg}[#1] \pushQED{\qed}}{\popQED \end{eg}}
\newtheorem{defnaux}[equation]{Definition}
\newenvironment{definition}[1][]{\begin{defnaux}[#1]\pushQED{\qed}}{\popQED \end{defnaux}}

\makeatletter
\renewcommand{\thesubsection}{%
  \ifnum\c@subsection<1 \@arabic\c@section
  \else \thesection.\@arabic\c@subsection
  \fi
}
\makeatother

\renewcommand{\theenumi}{\alph{enumi}}

\DeclareMathOperator{\soc}{soc}
\newcommand{\FI}{\mathbf{FI}}
\newcommand{\gen}{\mathrm{gen}}
\DeclareMathOperator{\Irr}{Irr}
\newcommand{\lf}{\mathrm{lf}}
 \title[Axiomatics for generic categories]{Axiomatics for generic categories in positive characteristic}

 \author{Karthik Ganapathy}
 \address{Department of Mathematics, University of California, San Diego, CA}
\email{\href{mailto:kganapathy@ucsd.edu}{kganapathy@ucsd.edu}}
\urladdr{\url{https://sites.google.com/view/karthik-ganapathy/}}

\begin{document}

\begin{abstract}
Generic categories arising in representation stability and equivariant commutative algebra are difficult to analyze in positive characteristic, as these categories rarely admit enough (or even any) finitely generated injectives. Our main observation is that they are nonetheless an increasing union of Serre subcategories, each with finitely many simples and enough finite length injectives. We package this into an axiomatic framework which is widely applicable. In particular, we use this to analyze (1) VI-modules in non-describing characteristic, recovering Nagpal's classification of simple generic VI-modules, and (2) GL-equivariant modules over truncated polynomial rings and exterior algebras in infinitely many variables.
\end{abstract}
\maketitle
\section{Introduction}\label{s:intro} 
A central goal of representation stability is the analysis of ``generic categories'' of module categories, i.e., their Serre quotient by the subcategory of torsion objects. The class of an object in the generic category retains only its generic (read asymptotic/dominating) behaviour. An early success in this regard was Sam--Snowden's description of the generic category of FI-modules over fields of characteristic zero \cite{ss16gl}: this category is locally artinian, and the simple objects are naturally indexed by partitions. Given a finitely generated FI-module $V$, the partitions indexing the simple constituents of its image in the generic category correspond precisely to the (padded) partitions appearing in the uniform decomposition of $V_n$ for $n \gg 0$ -- the phenomenon from which the subject takes its name.

In positive characteristic, it essentially follows from Harman's work \cite{har15per} that the generic category of FI-modules is locally artinian. Nagpal \cite{nag21vi} studied this category in detail, classifying the simple objects; in fact his results hold for the generic category of VI-modules in non-describing characteristic, of which FI-modules are a degeneration. The integral version was described combinatorially by Patzt--Wiltshire-Gordon \cite{pw24tails}, and Kriz computed examples of objects in the generic category exhibiting pathological behaviour away from characteristic zero \cite{kri22gen, kri23gen}. 
More recently, over fields of characteristic zero, Snowden \cite{sno21stable} extended many results about the generic category of FI-modules to an arbitrary finitely generated polynomial twisted commutative algebra; these generic categories are closely related to the representation theory of certain linear-oligomorphic groups that he introduced along with Harman \cite{hs24tensor} in the course of their groundbreaking program to construct new symmetric tensor categories.

This paper grew out of our desire to obtain a more transparent proof of the main theorems from Part II of Nagpal's work on the generic category of VI-modules in non-describing characteristic \cite{nag21vi}, ideally using only the results of Part~I \cite{nag19vi}. Nagpal establishes several strong structural results in Part~I already, so it seemed unlikely to us that Part~II had to be as technical as it is. This expectation was reinforced when Snowden's work \cite{sno21stable} appeared, since it uses only analogues of a fraction of the results Nagpal establishes in the VI-setting. 

\subsection{Results} Our main contribution is an axiomatic framework which is flexible enough to analyze many generic categories of interest in representation stability and equivariant commutative algebra, especially over fields of positive characteristic. This occupies Section~\ref{s:axioms}. The relevant results from Grothendieck abelian category theory we use in this section are collected in Appendix~\ref{s:appx}. The framework we develop allows us to easily recover Nagpal's results from Part~II \cite{nag21vi} in Section~\ref{s:vi}. 

To show the versatility of our theory, we also analyze the generic category of certain nilpotent $\GL$-algebras in positive characteristic. 
Specifically, we prove that generic category of $\GL$-equivariant modules over truncated polynomial rings and exterior algebras are locally artinian and classify the simple objects. The analogues of Nagpal's Part~I results required in this section were previously established by the author \cite{gan22ext, gan24monomial, gan25pol}. 

The difficulty with analyzing generic categories in positive characteristic stems from them lacking finitely generated injectives. Our key observation is that these generic categories are an increasing union of Serre subcategories, each of which has finitely many simples and admits enough finite length injectives. This is reminiscent of the algebraic representation theory of the orthogonal and symplectic groups, where finite-dimensional injectives are available only upon restricting to polynomial representations of bounded degree. We however caution that these generic categories are often not highest weight categories, in the sense of Cline--Parshall--Scott \cite{cps88high} or Brundan--Stroppel \cite{bs24semi}, as one easily sees for VI-modules in (positive) non-describing characteristic.

\subsection{Relation to other work}
Our results lie at the confluence of three historically distinct viewpoints on stability patterns in representation theory.
\subsubsection{Stable representation theory}
The generic category of FI-modules is equivalent to the category of smooth representations of $\fS_{\infty}$. Similarly, the generic category of VI-modules in non-describing characteristic is equivalent to the category of smooth representations of $\GL_{\infty}(\mathbf{F}_q)$. Both groups are \emph{oligomorphic}: a permutation group $G \subset \Sym(\Omega)$ is oligomorphic if for every $n \geq 0$ there is a finite subset $X_n \subset \Omega$ such that every $n$-element subset $Y \subset \Omega$ satisfies $gY \subset X_n$ for some $g \in G$; equivalently, $G$ has finitely many orbits on $\Omega^n$ for all $n$.

Harman and Snowden \cite{hs24tensor} recently introduced a linear analogue: a subgroup $G \subset \GL(k^{\infty})$ is \emph{linear-oligomorphic} if for every $n \geq 0$ there is a finite-dimensional subspace $W \subset k^{\infty}$ such that every $n$-dimensional subspace $U \subset k^{\infty}$ satisfies $gU \subset W$ for some $g \in G$. The motivating examples are stabilizers of tensors like the infinite-rank orthogonal group. The algebraic representation theory of linear-oligomorphic groups is well-behaved -- for instance, at least in characteristic zero, the category is locally artinian.

We hope that the framework of Section~\ref{s:axioms} is a step toward the representation theory of linear-oligomorphic groups in positive characteristic, in the spirit of \cite{sno21stable}. Even for the infinite-rank classical groups, the algebraic representation theory in positive characteristic appears not to have been developed, in contrast with characteristic zero, where it is understood from several points of view \cite{ss15stab, dps16koszul}. Our results in Section~\ref{s:gl} can be reinterpreted as analyzing (partial) Frobenius kernels of the infinite-rank general affine group.

\subsubsection{Generic representation theory}
The study of functors from finite-dimensional vector spaces over a finite field to vector spaces over a field $k$, known as generic representation theory, has advanced steadily for three decades \cite{kuh94gen, pow98ii, dja07fonc, kuhn15gen, dg24gen} allowing to streamline several important results about the Steenrod algebra proved in the 1980s (see \cite{kuh00gen}). One of its central problems, the Lannes--Schwartz artinian conjecture \cite[Conjecture~4.6]{kuh00gen}, was settled independently by Putman--Sam \cite{ps17duke} and Sam--Snowden \cite{ss17gro}. These proofs crucially use the category VI and in \cite{ss17gro}, also the category of finite sets and surjections. The sharper form of the conjecture \cite[Conjecture~6.8]{kuh00gen} predicting the Krull dimension of the projective objects remains open (see also \cite[Page~223]{pow00surv}). Although we only focus on the non-describing characteristic case, we expect the systematic study of functor categories to be a productive source of techniques, much as the passage to finite sets and surjections was for the Artinian conjecture.

\subsubsection{Global representation theory}
The papers \cite{ps22outer} and \cite{bbpsw26grt} initiate the study of representations of outer automorphism groups of compatible families of finite groups, partially motivated by equivariant homotopy theory. The category of VI-modules is a foundational example in this theory. The papers \cite{bbpsw25ttvi} and \cite{xu26vi} study the tensor-triangular geometry and tensor-abelian geometry of global representations, albeit over fields of characteristic zero. Extending their results to positive characteristic would be interesting. 

\begin{remark}
    A ubiquitous result in characteristic zero about the module categories we discussed above is that projective objects are injective; This comes into play in both \cite{sno21stable} and \cite{bbpsw26grt}. The failure of this property in positive characteristic is closely related to their generic categories being difficult to analyze. Our framework provides a partial replacement to this property: in the case of $\VI$-modules for instance, our main result easily implies that a projective object generated in degree $d$ is injective in the Serre subcategory generated by the torsion objects and modules generated in degree $\leq d$. 
\end{remark}

\subsection{Notation} 
For a positive integer $n$, we let $[n] = \{1,2,\ldots, n\}$. All tensor products $\otimes$ are over the ground field unless specified otherwise.

\subsection*{Acknowledgements} I thank Nate Harman and Andrew Snowden for helpful discussions.

\subsection*{Disclosure of LLM use} {I used a publicly available LLM for copyediting.}\footnote{The LLM flagged a subtle error I made: in my original proof of Lemma~\ref{lem:artinian}, I wrongly assumed that the localizing subcategory generated by finitely many artinian objects is locally artinian. This holds under a noetherian assumption on the ambient category by Proposition~\ref{prop:AppA} (see also Example~\ref{exm}), so I added this in Section~\ref{s:axioms} as Property (A11). It may be necessary to get around this assumption to study linear-oligomorphic groups in positive characteristic given prior results on non-noetherianity of $\GL$-algebras in positive characteristic \cite{gan24noe, gan25noe}; these results however do not preclude their generic categories from being noetherian, or even artinian.}

\section{Filtering an abelian category by Serre subcategories}\label{s:axioms}
Throughout this section, we fix a field $k$.

The framework introduced here is, to a large extent, inspired by a similar framework \cite[Section~4.1]{sno21stable} of Sam--Snowden, used in op.~cit.~to study generic categories of finitely generated polynomial $\GL$-algebras in characteristic zero.

\subsection{Framework}
Let $\cA$ be a $k$-linear Grothendieck abelian category.

Let $\cL = \bigsqcup_{n \in \bZ} \cL_n$ be a graded set; we assume $\cL_{n} = \emptyset$ for $n < 0$.
Given $x \in \cL$, let $|x|$ be the unique integer $n$ such that $x \in \cL_n$. Define $\cL_{\leq n} = \bigsqcup_{i \leq n} \cL_i$ and analogously the set $\cL_{< n}$. 

Assume that for each $\lambda \in \cL$ there are distinguished nonzero finitely generated objects $K_{\lambda} \subset I_{\lambda}$ in $\cA$.

\begin{definition}
For $n \in \bZ$, an object $C$ in $\cA$ is
\begin{itemize}
    \item   \emph{$K$-filtered of degree $n$} if there exists a finite filtration $0 = C^0 \subset C^1 \subset \ldots \subset C^t = C$ such that $C^{i}/C^{i-1} \cong K_{\lambda(i)}$ for some $\lambda(i) \in \cL_{n}$ for all $i$. 
    \item  \emph{$K$-filtered of degree $\leq n$} if there exists a finite filtration $0 = C^0 \subset C^1 \subset \ldots \subset C^t = C$ such that $C^{i}/C^{i-1} \cong K_{\lambda(i)}$ for some $\lambda(i) \in \cL_{\leq n}$ for all $i$.
    \item  \emph{$K$-filtered} if there exists $n \in \bN$ such that $C$ is $K$-filtered of degree $\leq n$.
    \item \emph{gr-subK of degree $\leq n$} (resp.~ $< n$) if there exists a finite filtration $0 = C^0 \subset C^1 \subset \ldots \subset C^t = C$ together with, for each $i$, an injection $C^{i}/C^{i-1} \hookrightarrow D^i$ where $D^i$ is $K$-filtered of degree $\leq n$ (resp.~$< n$). \qedhere
\end{itemize}
\end{definition}

\begin{lemma}\label{lem:closure}
 Any subobject of a $K$-filtered object of degree $\leq n$ is gr-subK of degree $\leq n$. Moreover, the extension of two objects which are $K$-filtered of degree $\leq n$ (resp.~degree $n$) is $K$-filtered of degree $\leq n$ (resp.~degree $n$). Similarly, the extension of two objects which are gr-subK of degree $\leq n$ is also gr-subK of degree $\leq n$.
\end{lemma}
\begin{proof}
This is clear.
\end{proof}

For $n \geq 0$, let $\cA_{\leq n} \subset \cA$ be the localizing subcategory of $\cA$ generated by the objects $K_{\lambda}$ with $\lambda \in \cL_{\leq n}$. Put $\cA_{\leq n} = 0$ for $n < 0$. 
The category $\cA_{\leq n}$ is Grothendieck abelian for all $n\geq 0$, being a localizing subcategory of $\cA$.

Consider the following properties:
{\renewcommand{\theenumi}{A\arabic{enumi}}\renewcommand{\labelenumi}{\textup{(\theenumi)}}
\begin{enumerate}
\item The set $\cL_n$ is finite for all $n \in \bN$.
\item For finitely generated objects $C$ and $D$, we have $\dim_k \ext^i_{\cA}(C,D) < \infty$ for $i \in \{0,1\}$. 
\item For $\lambda, \mu \in \cL$ with $|\lambda| < |\mu|$ we have $\ext^i_{\cA}(K_{\lambda}, K_{\mu}) = 0$ for all $i \geq 0$.
\item For $\lambda, \mu \in \cL$ with $|\lambda| \leq |\mu|$ we have $\ext^i_{\cA}(K_{\lambda}, I_{\mu}) = 0$ for all $i > 0$.
\item The quotient $I_{\lambda}/K_{\lambda}$ is $K$-filtered of degree $|\lambda|$.
\item For all $\lambda \in \cL$, the $k$-algebra $\End_{\cA}(I_{\lambda})$ is a local ring.
\item Every nonzero morphism $K_{\lambda} \to I_{\gamma}$ with $|\gamma| = |\lambda|$ is injective.
\item No morphism $K_{\lambda} \to C$ where $C$ is gr-subK of degree $< |\lambda|$ is injective.
\item Any morphism $K_{\lambda} \to I_{\mu}$ with $|\mu| = |\lambda|$ but $\mu \ne \lambda$ is zero.
\item Every finitely generated object $M$ embeds into a $K$-filtered object.
\item The category $\cA$ is locally noetherian.
\end{enumerate}
}

Assuming Property (A5) holds, we have that for $\lambda \in \cL_n$, the object $I_{\lambda}$ is $K$-filtered of degree $n$ so $I_{\lambda}$ is contained in $\cA_{\leq n}$. We note some other immediate consequences.


\begin{lemma}\label{lem:filtration1}
Assume Property (A3) holds. Let $C$ be a $K$-filtered object. There exists a filtration
    \begin{displaymath}
    0 = C^{\le -1} \subset C^{\le 0} \subset C^{\le 1} \subset \ldots \subset C^{\le t} = C
    \end{displaymath}
    such that for all $0 \leq i \leq t$, the quotient $C^{\le i}/C^{\le{i - 1}}$ is a $K$-filtered object of degree $i$.
\end{lemma}
\begin{proof}
Let $0 = C^0 \subset C^1 \subset C^2 \subset \ldots \subset C^s = C$ be a filtration of $C$ such that $C^i / C^{i-1} \cong K_{\mu(i)}$ for all $i \in [s]$. Assume there exists an index $j$ such that $\vert\mu(j)\vert > \vert\mu(j+1)\vert$. We then have a short exact sequence
\begin{displaymath}
    0 \to C^j/C^{j-1} \to C^{j+1}/C^{j-1} \to C^{j+1}/C^j \to 0.
\end{displaymath}
However, by property~(A3) we have $\ext^1_{\cA}(K_{\mu(j+1)}, K_{\mu(j)}) = 0$, so the short exact sequence splits. So we may write $C^{j+1}/C^{j-1} \cong C^{j}/C^{j-1} \oplus N/C^{j-1}$ for a subobject $C^{j-1} \subset N \subset C^{j+1}$ with the first summand isomorphic to $K_{\mu(j)}$ and the second summand isomorphic to $K_{\mu(j+1)}$. We may now define a new filtration
\begin{displaymath}
    0 = \widetilde{C}^{0} \subset \widetilde{C}^{1} \subset \ldots \subset \widetilde{C}^{s} = C
\end{displaymath}
with $\widetilde{C}^{i} = C^i$ for $i \ne j$ and $\widetilde{C}^{j} = N$. Iterating this, we obtain the result upon coarsening the filtration as necessary.
\end{proof}

\begin{lemma}\label{lem:auxiliary}
Assume Properties (A1)--(A11) hold. 
\begin{enumerate}
    \item Let $\lambda \in \cL$ and $C$ be a $K$-filtered object of degree $|\lambda|$. A non-injective map $\phi \colon K_{\lambda} \to C$ is zero. 
    \item Let $\lambda, \mu \in \cL$ satisfy $|\lambda| < |\mu|$. Then $\ext^i_{\cA}(K_{\lambda}, I_{\mu}) = 0$ for all $i \geq 0$. 
\end{enumerate}

\end{lemma}
\begin{proof}
   (a)  Let $0 = C^0 \subset C^1 \subset \ldots \subset C^t = C$ be a filtration such that $C^{i}/C^{i-1} \cong K_{\lambda(i)}$ with $\lambda(i) \in \cL_{|\lambda|}$ for all $i \in [t]$. Let $r$ be the minimal index such that $\im \phi \subset C^r$. Assume first that $r > 0$. In this case, the composition $K_{\lambda} \to C^r \to C^r/C^{r-1} \cong K_{\lambda(r)} \hookrightarrow I_{\lambda(r)}$ is nonzero and so injective by (A7) but we assumed $\phi$ is not injective, a contradiction. So $r= 0$ which implies $\phi = 0$, as required.
    
   (b) This is immediate from Properties (A3)--(A5) and d{\'e}vissage. 
\end{proof}

We introduce one final object before stating our main theorem. 
\begin{definition}\label{defn:HLambda}
For each $\lambda \in \cL$, let $H_{\lambda}$ be the intersection of the kernels of all morphisms $K_{\lambda} \to C$ where $C$ is a finitely generated object in $\cA_{\leq |\lambda|-1}$. 
\end{definition}
It is a priori not clear that $H_{\lambda} \ne 0$. Our main theorem is:
\begin{theorem}\label{thm:main} 
Assume Properties (A1)--(A11) are satisfied. For all $n \in \bN$, we have the following:
   \begin{enumerate}
        \item Every finitely generated object in $\cA_{\leq n}$ has finite length.
        \item The set $\{H_{\lambda}\}_{\lambda \in \cL_{\leq n}}$ is a complete set of pairwise non-isomorphic simples of $\cA_{\leq n}$.
        \item For $\lambda \in \cL_{\leq n}$, the object $K_{\lambda}/H_{\lambda}$ lies in $\cA_{\leq |\lambda|-1}$.
        \item For $\lambda \in \cL_{\leq n}$, the injective envelope $J_{\lambda, n}$ of $H_{\lambda}$ in $\cA_{\leq n}$ is of finite length.
        \item For $\lambda \in \cL_n$, we have $J_{\lambda, n} = I_{\lambda}$.
        \item For $\lambda \in \cL_{\leq n}$, the indecomposable injective $J_{\lambda, n}$ is gr-subK of degree $\leq n$.
    \end{enumerate}
\end{theorem}
In particular, we obtain:
\begin{corollary}\label{cor:main}
Assume Properties (A1)--(A11) hold. The category $\cA$ is locally artinian with
$\{H_{\lambda}\}_{\lambda \in \cL}$ forming a complete set of pairwise non-isomorphic simples of $\cA$.
\end{corollary}
\begin{proof}
By Property~(A10), every finitely generated object of $\cA$ lies in $\cA_{\leq n}$ for $n$ large enough. The claim thus follows from Property (A11) and parts (a) and (b) of the theorem.
\end{proof}

\subsection{Proof of Theorem~\ref{thm:main}} This subsection is dedicated to proving the above theorem. In particular, we assume Properties (A1)--(A11) hold.

The proof is by induction on $n$, with the $n <0$ case being trivially true. So assume we know that the conclusion of the theorem holds for all integers $< n$. We now prove it for the non-negative integer $n$ using a series of lemmas. 

Since $\cA_{\leq n-1}$ is closed under subobjects in $\cA_{\leq n}$, the objects $H_{\lambda}$ with $\lambda \in \cL_{<n}$ remain simple in $\cA_{\leq n}$. We first provide an alternative characterization of $H_{\lambda}$ for $\lambda \in \cL_n$.

\begin{lemma}\label{lem:notinj-lower}
   Let $\lambda \in \cL_n$. Let $C$ be a finitely generated object in $\cA_{\leq n-1}$ and $\phi \colon K_{\lambda} \to C$. The morphism $\phi$ is not injective.
\end{lemma}
\begin{proof}
    By part (a) for $\cA_{\leq n-1}$, the object $C$ is finite length, so it injects into a finite length injective $J$ of $\cA_{\leq n-1}$ by part (d) for $\cA_{\leq n-1}$. By part (f) for $\cA_{\le n-1}$, each indecomposable injective summand of $J$ is gr-subK of degree $\leq n-1$, so the object $J$ is gr-subK of degree $\leq n-1 < |\lambda|$ by Lemma~\ref{lem:closure}. If $\phi$ were injective, the composite $K_{\lambda} \to C \hookrightarrow J$ would be an injective morphism into a gr-subK object of degree $< |\lambda|$, contradicting Property~(A8).
\end{proof}

Assume $\gamma \in \cL_{< n}$ and $\lambda \in \cL_n$. By Property (A2) and part (d) for $\cA_{\leq n-1}$, the vector space $\Hom_{\cA}(K_{\lambda}, J_{\gamma, n-1})$ is finite-dimensional, so we obtain a map $K_{\lambda} \to \Hom_{\cA}(K_{\lambda}, J_{\gamma, n-1})^* \otimes J_{\gamma, n-1}$ and in turn, by Property (A1), a map 
   \[ \iota_{\lambda} \colon K_{\lambda} \to E \coloneqq \bigoplus_{\gamma \in \cL_{<n}} \Hom_{\cA}(K_{\lambda}, J_{\gamma, n-1})^* \otimes J_{\gamma, n-1}.\]
\begin{lemma}\label{lem:HLambda}
 We have $H_{\lambda} = \ker(\iota_{\lambda})$. 
\end{lemma}
\begin{proof}
Since $E$ is a finitely generated object in $\cA_{\leq n-1}$, we have $H_{\lambda} \subset \ker (\iota_{\lambda})$. 

Given a map $\phi \colon K_{\lambda} \to C$ with $C$ finitely generated in $\cA_{\leq n-1}$, we may choose an injective envelope of $C$ in $\cA_{\leq n-1}$ which is a finite direct sum of copies of $J_{\gamma, n-1}$ with $\gamma \in \cL_{< n}$ by part (a), (b) and (d) for $\cA_{\leq n-1}$. Thus, the map $\phi$ extends to a map $K_{\lambda} \to E$ so $\ker \phi \supset \ker (\iota_{\lambda})$. Since $\phi$ was arbitrary and $H_{\lambda}$ is the intersection of kernels of such maps, we get that $\ker (\iota_{\lambda}) \subset H_{\lambda}$. 
\end{proof}

Since $H_{\lambda} = \ker (\iota_{\lambda})$ and $E$ is finitely generated, we see that $H_{\lambda}$ is nonzero by Lemma~\ref{lem:notinj-lower}.

\begin{lemma}\label{lem:HLambdasimple}
    For $\lambda \in \cL_n$, the object $H_{\lambda}$ is simple.
\end{lemma}
\begin{proof}
    Assume $N \subset H_{\lambda}$ is a nonzero subobject. By Property~(A10), the finitely generated object $K_{\lambda}/N$ admits an injection $\phi \colon K_{\lambda}/N \hookrightarrow M$ where $M$ is $K$-filtered. By Lemma~\ref{lem:filtration1}, we may assume there are subobjects $C_1 \subset C_2 \subset M$ such that $C_1$ is $K$-filtered of degree $\leq n-1$, the object $C_2/C_1$ is $K$-filtered of degree $n$, and $M/C_2$ is filtered by $K_{\gamma}$ with $\gamma \in \cL \setminus \cL_{\leq n}$. We first show that $\im \phi \subset C_1$.

    By Property~(A3), we have $\Hom_{\cA}(K_{\lambda}, K_{\gamma}) = 0$ for $\gamma \in \cL \setminus \cL_{\leq n}$. So the composition $K_{\lambda}/N \to M \to M/C_2$ is the zero map hence the morphism $\phi$ factors as $\overline{\phi} \colon K_{\lambda}/N \hookrightarrow C_2$. Let $\psi \colon K_{\lambda} \to C_2/C_1$ be the composite of $K_{\lambda} \to K_{\lambda}/N \xrightarrow{\overline{\phi}} C_2$ with the projection $C_2 \to C_2/C_1$; its kernel contains $N \neq 0$ and the object $C_2/C_1$ is $K$-filtered of degree $n$, so $\psi = 0$ by Lemma~\ref{lem:auxiliary}(a), proving the claim.

    Thus the composite $K_{\lambda} \to K_{\lambda}/N \hookrightarrow C_1$ is a map from $K_{\lambda}$ to a finitely generated object in $\cA_{\leq n-1}$ so $H_{\lambda}$ is contained in its kernel, which is precisely $N$, whence $H_{\lambda}$ is simple.
\end{proof}

\begin{lemma}\label{lem:KLHLsimples}
  For $\lambda \in \cL_n$, the quotient $K_{\lambda}/H_{\lambda}$ is finite length and the simple objects occurring with nonzero multiplicity in $K_{\lambda}/H_{\lambda}$ are of the form $H_{\gamma}$ with $\gamma \in \cL_{< n}$. 
\end{lemma}
\begin{proof}
By Lemma~\ref{lem:HLambda}, we have an exact sequence
\[0 \to H_{\lambda} \to K_{\lambda} \to \bigoplus_{\gamma \in \cL_{<n}} \Hom_{\cA}(K_{\lambda}, J_{\gamma, n-1})^* \otimes J_{\gamma, n-1}.\] 
The result now follows from parts (d) and (b) for $\cA_{\leq n-1}$.
\end{proof}
\begin{corollary}\label{cor:KLfl}
   For $\lambda \in \cL_n$, the object $K_{\lambda}$ is finite length. 
\end{corollary}
\begin{proof}
    The object $K_{\lambda}$ is an extension of $H_{\lambda}$ and $K_{\lambda}/H_{\lambda}$ both of which are finite length by Lemma~\ref{lem:HLambdasimple} and Lemma~\ref{lem:KLHLsimples}.
\end{proof}

\begin{lemma}\label{lem:simpleclass}
    Every simple object in $\cA_{\leq n}$ is of the form $H_{\gamma}$ for some $\gamma \in \cL_{\leq n}$.
\end{lemma}
\begin{proof}
This is Lemma~\ref{lem:Appsimple} combined with Lemma~\ref{lem:KLHLsimples} and Lemma~\ref{lem:HLambdasimple}.
\end{proof}

\begin{lemma}\label{lem:artinian}
   Every finitely generated object in $\cA_{\leq n}$ has finite length.
\end{lemma}
\begin{proof}
    By Proposition~\ref{prop:AppA} and Property (A11), every finitely generated object in $\cA_{\leq n}$ is in the Serre subcategory generated by the objects $K_{\gamma}$ with $\gamma$ varying over $\cL_{\leq n}$. The generating objects are finite length by Corollary~\ref{cor:KLfl} (and induction), whence the Serre subcategory they generate is artinian.
\end{proof}

Lemma~\ref{lem:artinian} and Lemma~\ref{lem:KLHLsimples} give parts (a) and (c) for $\cA_{\leq n}$ respectively. We emphasize however that Lemma~\ref{lem:simpleclass} only proves one ``containment'' required for part (b) -- we are yet to prove $H_{\lambda} \not\cong H_{\gamma}$ for distinct $\lambda, \gamma \in \cL_{\leq n}$. 

\begin{lemma}\label{lem:Jlaminjenv}
   For $\lambda \in \cL_n$, the object $I_{\lambda}$ is an indecomposable injective in $\cA_{\leq n}$ and the injective envelope of $H_{\lambda}$.
\end{lemma}
\begin{proof}
  It suffices to prove that $\ext^1_{\cA_{\leq n}}(M, I_{\lambda}) = 0$ for all \emph{finitely generated} $M$ in $\cA_{\leq n}$ by Property (A11) and Proposition~\ref{prop:baer}. 
    
    We first show that $\ext^i_{\cA}(M, I_{\lambda}) = 0$ for all finitely generated $M$ in $\cA_{\leq n}$ and all $i > 0$. 
    We proceed by induction on $m$ to show that for all finitely generated objects $M$ in $\cA_{\leq m}$, we have $\ext^i_{\cA}(M, I_{\lambda}) = 0$ for all $i > 0$. The base case of $m < 0$ is trivially true. So assume we have $\ext^i_{\cA}(M, I_{\lambda}) = 0$ for all finitely generated objects $M$ in $\cA_{\leq m}$ for some $m < n$.
    
    Assume $m+1 < n$ (resp.~$m+1 = n$), then by parts (a) and (b) for $\cA_{\leq m+1}$ (resp.~part (a) for $\cA_{\leq n}$ and Lemma~\ref{lem:simpleclass}), it suffices to prove that $\ext^i_{\cA}(H_{\gamma}, I_{\lambda}) = 0$ for $\gamma \in \cL_{\leq m+1}$. When $|\gamma| \leq m$ this vanishing holds by induction, so we may assume $|\gamma| = {m+1}$. By part (c) for $\cA_{\leq m+1}$, we have the short exact sequence
    \[ 0 \to H_{\gamma} \to K_{\gamma} \to N \to 0 \]
    with $N$ an object of $\cA_{\leq m}$ so by induction, we have $\ext^i_{\cA}(N, I_{\lambda}) = 0$ for all $i > 0$. From the associated long exact sequence
    \[\begin{tikzcd}
0 & {\Hom_{\cA}(N, I_{\lambda})} & {\Hom_{\cA}(K_{\gamma}, I_{\lambda})} & {\Hom_{\cA}(H_{\gamma}, I_{\lambda})} \\
& {\ext_{\cA}^1(N, I_{\lambda})} & {\ext_{\cA}^1(K_{\gamma}, I_{\lambda})} & {\ext_{\cA}^1(H_{\gamma}, I_{\lambda})} \\
& {\ext_{\cA}^2(N, I_{\lambda})} & {\ext_{\cA}^2(K_{\gamma}, I_{\lambda})} & {\ext_{\cA}^2(H_{\gamma}, I_{\lambda})} \ldots,
\arrow[from=1-1, to=1-2]
\arrow[from=1-2, to=1-3]
\arrow[from=1-3, to=1-4]
\arrow[from=2-2, to=2-3]
\arrow[from=2-3, to=2-4]
\arrow[from=3-2, to=3-3]
\arrow[from=3-3, to=3-4]
\arrow[from=1-4, to=2-2, rounded corners,
  to path={ -- ([xshift=2ex]\tikztostart.east)
            |- ([xshift=-2ex, yshift=3mm]\tikztotarget.north west) \tikztonodes
            |- (\tikztotarget.west) }]
\arrow[from=2-4, to=3-2,  rounded corners,
  to path={ -- ([xshift=2ex]\tikztostart.east)
            |- ([xshift=-2ex, yshift=3mm]\tikztotarget.north west) \tikztonodes
            |- (\tikztotarget.west) }]
\end{tikzcd}\]
    we get $\ext_{\cA}^i(H_{\gamma}, I_{\lambda}) \cong \ext_{\cA}^i(K_{\gamma}, I_{\lambda})$ for all $i > 0$. But $\ext^i_{\cA}(K_{\gamma}, I_{\lambda}) = 0$ for all $i > 0$ by Property~(A4). We have thus proved our claim that $\ext^i_{\cA}(M, I_{\lambda}) = 0$ for $i > 0$ and finitely generated $M \in \cA_{\leq n}$.  In particular $\ext^1_{\cA_{\leq n}}(M, I_{\lambda}) \cong \ext^1_{\cA}(M, I_{\lambda}) = 0$ for all finitely generated $M$ in $\cA_{\leq n}$ by Lemma~\ref{lem:obvious}, so $I_{\lambda}$ is injective in $\cA_{\leq n}$.

    By Property~(A6), the algebra $\End_{\cA}(I_{\lambda})$ has no non-trivial idempotents, so $I_{\lambda}$ is indecomposable as well. Finally, the inclusions $H_{\lambda} \subseteq K_{\lambda} \subset I_{\lambda}$ give an injection $H_{\lambda} \hookrightarrow I_{\lambda}$; as $I_{\lambda}$ is an indecomposable injective, it is the injective envelope of $H_{\lambda}$, and $H_{\lambda} = \soc(I_{\lambda})$.
\end{proof}

So for $\lambda \in \cL_n$ we have $J_{\lambda, n} = I_{\lambda}$, thus proving part (e) for $\cA_{\leq n}$. This readily gives part (b) for $\cA_{\leq n}$ as well:
\begin{lemma}\label{lem:distinct}
    Assume for $\lambda, \gamma \in \cL_{\leq n}$, we have $H_{\lambda} \cong H_{\gamma}$. Then $\lambda = \gamma$.
\end{lemma}
\begin{proof}
By part (b) for $\cA_{\leq n-1}$, we may assume $|\lambda| = n$ so $|\gamma| \leq |\lambda|$. 
Since $J_{\lambda,n}$ is the injective envelope of $H_{\lambda} \cong H_{\gamma}$ in $\cA_{\leq n}$, the containment $H_{\gamma} \hookrightarrow K_{\gamma}$ extends to a nonzero morphism $K_{\gamma} \to J_{\lambda, n} = I_{\lambda}$. The $i=0$ case of Lemma~\ref{lem:auxiliary}(b) now forces $|\gamma| = |\lambda|$ and in turn, Property (A9) forces $\gamma = \lambda$, as required.
\end{proof}
So what remains to be proved are parts (d) and (f).

\begin{lemma}\label{lem:Ilower-finite}
For all $\gamma \in \cL_{\leq n}$, the injective envelope $J_{\gamma, n}$ of $H_{\gamma}$ in $\cA_{\leq n}$ is finite length.
\end{lemma}
\begin{proof}
If $\gamma \in \cL_n$, then $J_{\gamma, n} = I_{\gamma}$ by Lemma~\ref{lem:Jlaminjenv}, which is finitely generated and hence finite length by Lemma~\ref{lem:artinian}. So for the remainder we may assume $\gamma \in \cL_{<n}$.

Let $G \colon \cA_{\leq n} \to \cA_{\leq n-1}$ be the right adjoint to the inclusion functor $\cA_{\leq n-1} \to \cA_{\leq n}$. The functor $G$ sends an object to the largest subobject lying in $\cA_{\leq n-1}$ (which exists since $\cA_{\leq n-1}$ is closed under colimits). Furthermore, $G$ takes injectives to injectives as its left adjoint is exact.
So $G(J_{\gamma, n})$ is an injective object of $\cA_{\leq n-1}$ with socle $H_{\gamma}$ which implies $G(J_{\gamma, n}) \cong J_{\gamma, n-1}$ and so by part (d) for $\cA_{\leq n-1}$, this subobject is finite length. Let $C = \coker(J_{\gamma, n-1} \to J_{\gamma, n})$; we show that $C$ is finite length.

Note that since we have proved $\cA_{\leq n}$ is locally artinian, the socle of an object is an essential subobject. So we first compute the socle of $C$. Let $\delta \in \cL_{\leq n}$. The long exact sequence associated to $\Hom_{\cA_{\leq n}}(H_{\delta}, -)$ is
\[\begin{tikzcd}
& {\Hom_{\cA_{\leq n}}(H_{\delta}, J_{\gamma, n-1})} & {\Hom_{\cA_{\leq n}}(H_{\delta}, J_{\gamma, n})} & {\Hom_{\cA_{\leq n}}(H_{\delta}, C)} \\
& {\ext^1_{\cA_{\le n}}(H_{\delta}, J_{\gamma, n-1})} & {\ext^1_{\cA_{\le n}}(H_{\delta}, J_{\gamma, n}) = 0;}
\arrow[from=1-2, to=1-3]
\arrow[from=1-3, to=1-4]
\arrow[from=2-2, to=2-3]
\arrow[from=1-4, to=2-2, rounded corners,
  to path={ -- ([xshift=2ex]\tikztostart.east)
            |- ([xshift=-2ex, yshift=3mm]\tikztotarget.north west) \tikztonodes
            |- (\tikztotarget.west) }]
\end{tikzcd}\]
the last term is $0$ because $J_{\gamma, n}$ is injective in $\cA_{\leq n}$. For $\delta \in \cL_{< n}$, we have $\ext^1_{\cA_{\le n}}(H_{\delta}, J_{\gamma, n-1}) \cong \ext^1_{\cA_{\leq n-1}}(H_{\delta}, J_{\gamma, n-1}) = 0$ since $J_{\gamma, n-1}$ is injective in $\cA_{\leq n-1}$; in this case, the first map in the long exact sequence is an isomorphism so $\Hom_{\cA_{\leq n}}(H_{\delta}, C) = 0$. For $\delta \in \cL_n$, the term $\Hom_{\cA_{\leq n}}(H_{\delta}, J_{\gamma, n})$ vanishes because $\soc(J_{\gamma, n}) = H_{\gamma}$ and $\delta \neq \gamma$, so $\Hom_{\cA_{\leq n}}(H_{\delta}, C) \cong \ext^1_{\cA_{\leq n}}(H_{\delta}, J_{\gamma, n-1}) \cong \ext^1_{\cA}(H_{\delta}, J_{\gamma, n-1})$, which is finite-dimensional by Property~(A2). So $\soc(C)$ is a finite direct sum of copies of $H_{\delta}$ with $\delta$ varying over $\cL_n$, so by Property (A1), the socle is finite length. Since $\soc(C)$ is essential in $C$, by Lemma~\ref{lem:Jlaminjenv}, the object $C$ embeds into a finite direct sum of copies $J_{\delta, n}$ with $\delta \in \cL_n$; each such $J_{\delta, n}$ is finite length by the first paragraph, so $C$ is finite length, thus so is $J_{\gamma, n}$.
\end{proof}
The final part of the theorem follows from the above proof.
\begin{corollary}\label{cor:grsubK}
    For all $\gamma \in \cL_{\leq n}$, the injective object $J_{\gamma, n}$ is gr-subK of degree $\leq n$.
\end{corollary}
\begin{proof}
    First assume $\gamma \in \cL_n$. Then by Property~(A5) and Lemma~\ref{lem:Jlaminjenv}, the object $J_{\gamma, n} = I_{\gamma}$ is $K$-filtered of degree $\leq n$, being an extension of the $K$-filtered object $I_{\gamma}/K_{\gamma}$ by $K_{\gamma}$. In particular, it is gr-subK of degree $\leq n$.

   Assume instead that $\gamma \in \cL_{<n}$, and consider the short exact sequence $0 \to J_{\gamma, n-1} \to J_{\gamma, n} \to C \to 0$ from the previous proof. By part (f) for $\cA_{\leq n-1}$, the object $J_{\gamma, n-1}$ is gr-subK of degree $\leq n-1$, hence of degree $\leq n$. As we showed in the previous proof, the object $C$ embeds into a finite direct sum $\bigoplus_{\delta \in \cL_n} I_{\delta}$, which is $K$-filtered of degree $\leq n$ by Lemma~\ref{lem:closure} and the previous paragraph; hence $C$ is gr-subK of degree $\leq n$. Another application of Lemma~\ref{lem:closure} yields the result. 
\end{proof}

To summarize: for $\cA_{\leq n}$, we have part (a) in Lemma~\ref{lem:artinian}, part(b) in Lemma~\ref{lem:distinct} and Lemma~\ref{lem:simpleclass}, part (c) in Lemma~\ref{lem:KLHLsimples}, part (d) in Lemma~\ref{lem:Ilower-finite}, part (e) in Lemma~\ref{lem:Jlaminjenv}, and part (f) in Corollary~\ref{cor:grsubK}, respectively. This concludes the proof by induction.

\subsection{Auxiliary results}
We deduce some corollaries of Theorem~\ref{thm:main} so we assume we are in the setting of the theorem.
\begin{lemma}\label{lem:Kembed}
    Let $M$ be a finitely generated object in $\cA_{\leq n}$. Let $N$ be a $K$-filtered object with a filtration $0 = N^0 \subset N^1 \subset \ldots \subset N^r = N$ such that the successive quotients are of the form $K_{\lambda(i)}$ with $|\lambda(i)|>n$ for all $i \in [r]$. Then $\Hom_{\cA}(M, N) = 0$.
\end{lemma}
\begin{proof}
   First assume $r = 1$. Then the socle of $N$ is $H_{\lambda(1)}$ with $|\lambda(1)| > n$.  So any map $M \to N$ is zero as the image, if nonzero, must contain $H^{\lambda(1)}$, but since $M$ lies in $\cA_{\leq n}$, all its simple constituents are of the form $H_{\gamma}$ with $|\gamma|\leq n$. The result easily follows by induction on $r$.
\end{proof}

\begin{corollary}\label{cor:Kembed}
    Let $M$ be a finitely generated object in $\cA_{\leq n}$. Then $M$ embeds into a $K$-filtered object of degree $\leq n$.
\end{corollary}
\begin{proof}
    By Property (A10), the object $M$ embeds into a $K$-filtered object $N$. Call this injection $\phi$.
    By Lemma~\ref{lem:filtration1}, there exists a subobject $N' \subset N$ such that $N'$ is $K$-filtered of degree $\leq n$ and $N/N'$ has a filtration such that the successive quotients are of the form $K_{\lambda}$ with $|\lambda| > n$. The image of the map $\phi$ must be contained in $N'$: otherwise, we obtain a nonzero map $M \to N/N'$ contradicting Lemma~\ref{lem:Kembed}. The induced map $\overline{\phi} \colon M \to N'$ satisfies the requisite property.
\end{proof}

\begin{corollary}\label{cor:Kembed2}
Let $\lambda \in \cL_n$. 
The object $H_{\lambda}$ is the intersection of the kernels of all maps $K_{\lambda}$ to $K$-filtred objects of degree $< n$.
\end{corollary}
\begin{proof}
    The object $H_{\lambda}$ is defined to be the intersection of kernels of maps $K_{\lambda} \to M$ as $M$ varies over finitely generated objects in $\cA_{\leq n-1}$ 
    But by Corollary~\ref{cor:Kembed}, for any finitely generated object $M$, there exists an injective map $M \to M'$ with $M'$ being a $K$-filtered object of degree $< n$. The result follows by combining the previous two statements.
\end{proof}

\section{VI-modules in non-describing characteristic}\label{s:vi}
Throughout this section, we fix a finite field $F$ of characteristic $p$ with $q$ elements and let $k$ be an arbitrary field of characteristic $\ne p$.

For a category $\cD$, a \textit{$\cD$-module} is a functor $\cD \to \Mod_k$. The category of $\cD$-modules with morphisms being natural transformations is denoted $[\cD,\Mod_k]$; this is a $k$-linear Grothendieck abelian category.

We let $\VI_F$ be the category of finite-dimensional vector spaces over $F$ with injective linear maps, and $\VB_F$ be the subcategory of $\VI_F$ with only invertible maps. 
Let $[\VI_F, \Mod_k]^{\gen} = [\VI_F, \Mod_k]/[\VI_F, \Mod_k]^{\lf}$ where $[\VI_F, \Mod_k]^{\lf}$ is the localizing subcategory of locally finite-length objects. 

For the remainder of this section, we let $\cB = [\VI_F, \Mod_k]$, $\VBm = [\VB_F, \Mod_k]$, and $\VIg = [\VI_F, \Mod_k]^{\gen}$. We shall supress the field $F$ from $\VI_F$ (so a $\VI$-module means an object in $\cB$) and also from $\VB_F$.

The left adjoint to the forgetful functor $\cB \to \VBm$ is denoted 
\[\cI \colon \VBm \to \cB.\]
The right adjoint to the exact functor $T \colon \cB \to \VIg$ is $S \colon \VIg \to \cB$. These composition $TS \cong \id_{\VIg}$.

A $\VI$-module is \emph{induced} if it is finitely generated and in the essential image of $\cI$. A $\VI$-module is \emph{semi-induced} if it has a finite filtration such that the successive quotients are induced.

\subsection{Past results}
We recall some important results about $\VI$-modules, mostly due to Nagpal \cite{nag19vi}.
\begin{theorem}\label{thm:noetherian}
    The category $\cB$ is locally noetherian.
\end{theorem}
\begin{proof}
    This was proved in independent works of Sam--Snowden \cite[Corollary~8.3.3]{ss17gro} and Putman--Sam \cite[Theorem~A]{ps17duke}.
\end{proof}

\begin{lemma}\label{lem:fext}
    Assume $M$ and $N$ are finitely generated $\VI$-modules. For all $i \geq 0$, we have $\dim_k(\ext^i_{\cB}(M, N)) < \infty$.
\end{lemma}
\begin{proof}
   Since $M$ is finitely generated and $\cB$ is locally noetherian (Theorem~\ref{thm:noetherian}), there exists a projective resolution 
   \[
   \ldots \to P_1 \to P_0 \to M \to 0
   \]
   with $P_i$ being a finitely generated projective $\VI$-module for $i \geq 0$.
   Applying the functor $\Hom_{\cB}(-, N)$ to the above, we obtain a complex
   \[
    0 \to \Hom_{\cB}(P_0, N) \to \Hom_{\cB}(P_1, N) \to \ldots
   \]
   whose homology is $\ext^i_{\cB}(M, N)$. Each term of the above complex is a finite-dimensional vector space over $k$, whence the result follows.
\end{proof}

\begin{theorem}[Embedding theorem] \label{thm:embedding}
   Assume $C$ is a finitely generated object in $\VIg$. There exists an injection $C \hookrightarrow D$ where  $D = T(M)$ is the image under $T$ of a finitely generated semi-induced $\VI$-module $M$.
\end{theorem}
\begin{proof}
    Let $N$ be a finitely generated $\VI$-module such that $T(N) \cong C$. By the shift theorem \cite[Theorem~4.38]{nag19vi}, we have a map $N \to M$ where $M$ is a finitely generated semi-induced module with kernel being torsion. Applying the functor $T$, we thus get an injection $T(N) \cong C \to D \coloneqq T(M)$ as required.
\end{proof}

\begin{proposition} \label{prop:inducedsat}
Assume $M$ is a semi-induced $\VI$-module. The map $M \to ST(M)$ is an isomorphism and $\rR^iS(T(M)) = 0$ for all $i > 0$.   

Conversely, assume $M$ is a finitely generated $\VI$-module such that $M \to ST(M)$ is an isomorphism and $\rR^iS(T(M)) = 0$ for all $i > 0$. Then $M$ is semi-induced.
\end{proposition}
\begin{proof}
    The first part is \cite[Corollary~4.22]{nag19vi} and the second part is \cite[Theorem~4.36]{nag19vi}.
\end{proof}

\begin{theorem}[Finiteness of $\rR S$] \label{thm:finiteness}
   Assume $C$ is a finitely generated object in $\VIg$. The $\VI$-module $\rR^iS (C)$ is finitely generated for all $i \geq 0$ and vanishes for $i \gg 0$.
\end{theorem}
\begin{proof}
    This is \cite[Theorem~5.7]{nag19vi}.
\end{proof}

This next result will not be required to verify Property (A1)--(A11) but will be used to recover some results of Nagpal on Castelnuovo--Mumford regularity.

\begin{theorem}\label{thm:glr}
    Assume $M$ is a finitely generated $\VI$-module such that $\dim_k M(F^n) = P(q^n)$ for $n \gg 0$ where $P \in \bQ[T]$ is a polynomial of degree $d$. Further assume the unit map $M \to ST(M)$ is an isomorphism. Then the Castelnuovo--Mumford regularity of $M$ is $\leq 2 d$.
\end{theorem}
\begin{proof}
   This is \cite[Theorem~1.1(a)]{gl20bounds} combined with \cite[Theorem~5.1]{nag19vi}.
\end{proof}

\subsection{Verifying Properties (A1)--(A11)} In this subsection, we explain how to apply the framework from Section~\ref{s:axioms} to $\VI$-modules in non-describing characteristic. 

For $n \geq 0$, let $\cL_n = \Irr(k[\GL_n(F)])$ be the set of irreducible  $k[\GL_n(F)]$-modules. For $\Lambda \in \cL_n$, let $K_{\Lambda} = T(\cI(\Lambda))$ and $I_{\Lambda} = T(\cI(\Theta_{\Lambda}))$ where $\Theta_{\Lambda}$ is the injective envelope of $\Lambda$ in $\Rep_k(\GL_n(F))$. Since $\cI$ is exact, we have $K_{\Lambda} \subset I_{\Lambda}$. By Proposition~\ref{prop:inducedsat}, an object $M$ in $\VIg$ is $K$-filtered if and only if $S(M)$ is semi-induced.

\begin{lemma}\label{lem:gss}
   Let $N$ be an object in $\VIg$ such that $S(N)$ is semi-induced and $M$ a $\VI$-module. We have isomorphisms $\ext^i_{\cB}(M, S(N)) \cong \ext^i_{\VIg}(T(M), N)$ for all $i \geq 0$.
\end{lemma}
\begin{proof}
Let $N \to I^{\bullet}$ be an injective resolution. Applying $S$, we obtain a complex $S(N) \to S(I^{\bullet})$ which is exact as $\rR^iS(N)=0$ by Proposition~\ref{prop:inducedsat}. Furthermore, the functor $S$ takes injectives to injectives being right adjoint to the exact functor $T$. Thus, to compute $\ext^i_{\cB}(M, S(N))$, we may apply the functor $\Hom_{\cB}(M, -)$ to the complex $S(I^{\bullet})$ and take homology. However, using the adjunction $T \dashv S$, this is the same as applying $\Hom_{\VIg}(T(M), -)$ to $I^{\bullet}$ and taking homology. This latter procedure computes $\ext^i_{\VIg}(T(M), N)$. The result follows.
\end{proof}

\begin{corollary}\label{cor:adjunction}
Let $W$ be a $k[\GL_n(F)]$-module and $N$ a semi-induced $\VI$-module. For $i \geq 0$, we have isomorphisms
  \[\ext^i_{\VIg}(T(\cI(W)), T(N)) \cong \ext^i_{k[\GL_n(F)]}(W, N(F^n)).\] 
\end{corollary}
\begin{proof}
We have a natural isomorphism $\Hom_{\cB}(\cI(W), N) \cong \Hom_{\VBm}(W, N)$ and since the forgetful functor is exact, this induces isomorphisms $\ext^i_{\cB}(\cI(W), N) \cong \ext^i_{\VBm}(W, N) =\ext^i_{k[\GL_n(F)]}(W, N(F^n))$. 
Combining Lemma~\ref{lem:gss} with Proposition~\ref{prop:inducedsat}, we also get isomorphisms $\ext^i_{\cB}(\cI(W), N) \cong \ext^i_{\VIg}(T(\cI(W)), T(N))$ for all $i \geq 0$.
The result follows.
\end{proof}

\begin{corollary}\label{cor:extvanishing}
    Let $\Lambda, \Gamma \in \cL$ satisfy $n = |\Lambda| < |\Gamma|$. For all $i\geq 0$, we have
    \[\ext^i_{\VIg}(K_{\Lambda}, K_{\Gamma}) = 0.\] 
\end{corollary}
\begin{proof}
By Corollary~\ref{cor:adjunction}, we have $\ext^i_{\VIg}(K_\Lambda, K_\Gamma) \cong \ext^i_{k[\GL_n(F)]}(\Lambda, \cI(\Gamma)(F^n))$. But $\cI(\Gamma)(F^n) = 0$ for degree reasons.
\end{proof}

\begin{corollary}\label{cor:extvanishing2}
    Let $\Lambda, \Gamma \in \cL$ satisfy $n = |\Lambda| \leq |\Gamma|$. For all $i > 0$, we have
    \[\ext^i_{\VIg}(K_{\Lambda}, I_{\Gamma}) = 0.\] 
\end{corollary}
\begin{proof}
By Corollary~\ref{cor:adjunction}, we have $\ext^i_{\VIg}(K_\Lambda, I_\Gamma) \cong \ext^i_{k[\GL_n(F)]}(\Lambda, \cI(\Theta_\Gamma)(F^n))$. The higher Exts vanish because $\cI(\Theta_\Gamma)(F^n) = 0$ when $|\Gamma| > |\Lambda|$ and is an injective $k[\GL_n(F)]$-module when $|\Gamma| = |\Lambda|$.
\end{proof}

\begin{lemma}\label{lem:ext1}
    For finitely generated $C$ and $D$ in $\VIg$, we have 
    \[\dim_k(\Hom_{\VIg}(C, D)) < \infty\] 
    and
    \[\dim_k(\ext^1_{\VIg}(C, D)) < \infty.\]
\end{lemma}
\begin{proof}
Let $M = S(C)$ and $N = S(D)$, so that $T(M) \cong C$ and $T(N) \cong D$. Since $C$ and $D$ are finitely generated, the $\VI$-modules $M$ and $N$ are finitely generated. Furthermore, we have isomorphisms $\Hom_{\cB}(M, N) \cong \Hom_{\VIg}(C, D)$. The first part now follows by Lemma~\ref{lem:fext}.

For the second part, consider the two functors $S \colon \VIg \to \cB$ and $\Hom_{\cB}(M,-) \colon \cB \to \Mod_k$. Their composition is the functor $\Hom_{\cB}(M, S(-))$ which is naturally isomorphic to the functor $\Hom_{\VIg}(C, -)$ since $T \dashv S$. So for $i \geq 0$, we have natural isomorphisms
\[\rR^i (\Hom_{\cB}(M,S(-))) \cong \ext^i_{\VIg}(C, -).\]
Being right adjoint to the exact functor $T$, the functor $S$ takes injectives in $\VIg$ to injectives in $\cB$. So the Grothendieck spectral sequence yields the exact sequence
   \[
   \ext^1_{\cB}(M, S(D)) \to \ext^1_{\VIg}(C, D) \to \Hom_{\cB}(M, \rR^1S(D)).
   \] 
 The $\VI$-modules $M, S(D)$, and $\rR^1S(D)$ are all finitely generated by Theorem~\ref{thm:finiteness}, so by Lemma~\ref{lem:fext}, the outermost vector spaces in the above sequence are finite-dimensional and so the middle vector space is also finite-dimensional, as required.
\end{proof}

\begin{lemma}\label{lem:qpol}
    For $\Lambda \in \cL_d$, there exists a polynomial $P \in \bQ[T]$ of degree $d$ such that $\dim_k(\cI(\Lambda)(F^n)) = P(q^n)$ for $n \gg 0$. 
\end{lemma}
\begin{proof}
   For $n \geq d$, we have $\cI(\Lambda)(F^n) = k[\Hom_{\VI}(F^d, F^n)] \otimes_{k[\GL_d(F)]} \Lambda$ which has dimension 
   \[\frac{(q^n-1)(q^n-q)\cdots (q^n - q^{d-1})}{(q^d-1)(q^d - q) \cdots (q^d - q^{d-1})} \dim_k \Lambda,\] from which the result follows.
\end{proof}
For a $\VB$-module $W$, we let $\max\deg W$ be the infimum of the set of integers $n$ such that $W_{n+i} = 0$ for all positive integers $i$. 
\begin{corollary}\label{cor:qpol2}
    For a nonzero finite-length $\VB$-module $W$, there exists a polynomial $P \in \bQ[T]$ of degree $\max \deg(W)$ such that $\dim_k(\cI(W)(F^n)) = P(q^n)$ for $n \gg 0$. 
\end{corollary}
\begin{proof}
    This is clear by using Lemma~\ref{lem:qpol} and induction on the length of $W$.
\end{proof}

\begin{corollary}\label{cor:qpol}
    Assume $C$ is a gr-subK object of degree $\leq d$ in $\VIg$. There exists a polynomial $P \in \bQ[T]$ of degree $\leq d$ such that $\dim_k(S(C)(F^n)) \leq P(q^n)$ for $n \gg 0$.
\end{corollary}
\begin{proof}
Let $0 = C^0 \subset C^1 \subset \ldots \subset C^t = C$ be a filtration, with injections $\iota_i \colon C^{i+1}/C^i \hookrightarrow K_{\Lambda(i)}$ where each $\Lambda(i) \in \cL_{\leq d}$.

Fix $i < t$. Applying the left exact functor $S$ to $\iota_i$ gives an injection $S(C^{i+1}/C^i) \hookrightarrow S(K_{\Lambda(i)})$. Since $\cI(\Lambda(i))$ is induced, we have $S(K_{{\Lambda}(i)}) \cong \cI({\Lambda}(i))$. By Corollary~\ref{cor:qpol2}, there exists a polynomial $Q_i \in \bQ[T]$ of degree $\leq d$ such that $\dim_k(\cI({\Lambda}(i))(F^n)) = Q_i(q^n)$ for $n \gg 0$. In particular, we have
\[
  \dim_k S(C^{i+1}/C^i)(F^n) \leq \dim_k \cI({\Lambda}(i))(F^n) = Q_i(q^n).
\]

Now applying $S$ to $0 \to C^i \to C^{i+1} \to C^{i+1}/C^i \to 0$ and evaluating at $F^n$ yields the exact sequence
\[
  0 \to S(C^i)(F^n) \to S(C^{i+1})(F^n) \to S(C^{i+1}/C^i)(F^n),
\]
whence $\dim_k S(C^{i+1})(F^n) \leq \dim_k S(C^i)(F^n) + \dim_k S(C^{i+1}/C^i)(F^n)$. As $S(C^0) = S(0) = 0$, induction on $i$ together with the above inequality yields
\[
  \dim_k S(C)(F^n)
  \leq \sum_{i=0}^{t-1} \dim_k S(C^{i+1}/C^i)(F^n)
  \leq \sum_{i=0}^{t-1} Q_i(q^n) = P(q^n),
\]
where $P := \sum_{i=0}^{t-1} Q_i \in \bQ[T]$ has degree $\leq d$.
\end{proof}

\begin{lemma}\label{lem:noinjective}
    Assume $C$ is a gr-subK object of degree $<d$ in $\VIg$. For $\Lambda \in \cL_d$, there are no injective maps $K_{\Lambda} \to C$.
\end{lemma}
\begin{proof}
    By applying $S$ to a map $\phi \colon K_{\Lambda} \to C$, we obtain a map $S(\phi) \colon S(K_{\Lambda}) \cong \cI(\Lambda) \to S(C)$. By Lemma~\ref{lem:qpol}, there exists a polynomial $P_1$ of degree $d$ such that $\dim_k(\cI(\Lambda)(F^n)) = P_1(q^n)$ for $n \gg 0$, and similarly by Corollary~\ref{cor:qpol}, there exists a polynomial $P_2$ of degree $< d$ such that $\dim_k(S(C)(F^n)) \le P_2(q^n)$ for all $n \gg 0$. Thus $S(\phi)$ cannot be injective and since $S$ is left exact, neither can $\phi$.
\end{proof}

\begin{lemma}\label{lem:nononinj}
    Let $\phi\colon K_{\Lambda} \to I_{\Gamma}$ be a nonzero map with $|\Gamma| = |\Lambda|$. Then (1) $\Gamma = \Lambda$ and (2) $\phi$ is injective.
\end{lemma}
\begin{proof}
    Applying $S$, we obtain a map $S(\phi) \colon \cI(\Lambda) \to \cI(\Theta_{\Gamma})$. By adjunction, we have an isomorphism $\Hom_{\cB}(\cI(\Lambda), \cI(\Theta_{\Gamma})) \cong \Hom_{\VBm}(\Lambda, \Theta_{\Gamma})$. Any nonzero map $\Lambda \to \Theta_{\Gamma}$ must be injective as $\Lambda$ is irreducible and this forces $\Lambda = \Gamma$ since $\soc(\Theta_{\Gamma}) = \Gamma$.
    By exactness of $\cI$, any nonzero $\cI(\Lambda) \to \cI(\Theta_{\Gamma})$ is also injective and in turn, the map $\phi$ is injective as $TS \cong \id_{\VIg}$.
\end{proof}

\begin{proposition}\label{prop:propvi}
    Properties (A1)--(A11) hold for $\VIg$. 
\end{proposition}
\begin{proof}
\leavevmode
{\renewcommand{\theenumi}{A\arabic{enumi}}\renewcommand{\labelenumi}{\textup{(\theenumi)}}
\begin{enumerate}
\item clearly holds as the number of irreducible representations of a finite group is finite.
\item is Lemma~\ref{lem:ext1}.
\item is Corollary~\ref{cor:extvanishing}. 
\item is Corollary~\ref{cor:extvanishing2}.
\item holds since $\Theta_{\Lambda}$ has a filtration where the successive quotients are irreducible $\GL_{|\Lambda|}(F)$-modules and the functors $T$ and $\cI$ are exact.
\item holds since $\End_{\VIg}(I_{\Lambda}) \cong \End_{k[\GL_n(F)]}(\Theta_{\Lambda})$ by Corollary~\ref{cor:adjunction} which is a local ring since $\Theta_{\Lambda}$ is an indecomposable $k[\GL_n(F)]$-module being the injective envelope of $\Lambda$ by construction.
\item is Lemma~\ref{lem:nononinj}(2).
\item is Lemma~\ref{lem:noinjective}.
\item is Lemma~\ref{lem:nononinj}(1).
\item holds by Theorem~\ref{thm:embedding} as the image of a semi-induced object in $\VIg$ is $K$-filtered.
\item follows by Theorem~\ref{thm:noetherian}. \qedhere
\end{enumerate}}
\end{proof}

\subsection{Recovering Nagpal's results}
We now recover the main results from \cite{nag21vi}.
Given $\Lambda \in \Irr(k[\GL_d(F)])$, let $\cH(\Lambda)$ be the intersection of the kernels of all maps $\cI(\Lambda) \to M$ as $M$ varies over all finitely generated semi-induced $\VI$-modules generated in degree $< d$.
The next result was first proved by Nagpal \cite[Theorems 1.1, 1.3]{nag21vi}. We easily recover it using our framework.

\begin{theorem}\leavevmode
   \begin{enumerate}
   \item $\VIg$ is locally artinian.
   \item There is a bijection 
   \[
    H \colon \bigsqcup_{n \in \bN} \Irr(k[\GL_n(F)]) \to \Irr(\VIg)
   \]
   where under this correspondence $H(\Lambda) = \soc T(\cI(\Lambda))$ in $\VIg$ for $\Lambda \in \Irr(k[\GL_d(F)])$.
   \item We have $S(H(\Lambda)) = \cH(\Lambda)$ for all $\Lambda \in \cL$.
   \item The Castelnuovo--Mumford regularity of $\cH(\Lambda)$ is at most $2 |\Lambda|$ for all $\Lambda \in\cL$.
   \end{enumerate}
\end{theorem}
\begin{proof} \leavevmode
\begin{enumerate}
    \item This is the first part of Corollary~\ref{cor:main}, which applies by Proposition~\ref{prop:propvi}.
    \item This is second part Corollary~\ref{cor:main} with $H(\Lambda) \coloneqq H_{\Lambda}$.
    \item By Corollary~\ref{cor:Kembed2}, the object $H(\Lambda)$ is the intersection of kernels of all maps $K_{\Lambda} = T(\cI(\Lambda)) $ to $K$-filtered object of degree $<d$ in $\VIg$. It follows from the left exactness of $S$ that $S(H(\Lambda))$ is precisely the intersection of kernels of maps $\cI(\Lambda) $ to semi-induced modules of degree $< d$, which we defined to be $\cH(\Lambda)$. 
    \item By an easy induction, we have that $\dim_k (\cH(\Lambda)(F^n))$ grows to the order $q^{|\Lambda| n}$. The previous part now shows that $\cH(\Lambda)$ is saturated. Theorem~\ref{thm:glr}  now applies.\qedhere
\end{enumerate}
\end{proof}

\begin{remark}
    In characteristic zero, the category of $\VI$-modules was studied in detail by Gan--Watterlond \cite{gw18rep} (see also \cite{gw18stable}).
\end{remark}

As with Nagpal's work \cite{nag21vi}, almost all of the above results also apply to $\FI$-modules in positive characteristic. We leave this easy extension to the reader.

\section{Some nilpotent GL-algebras in positive characteristic} \label{s:gl}
Throughout this section, we fix an infinite field $k$ of characteristic $p > 0$ and let $q = p^r$ with $r \geq 1$.

We fix a countable infinite-dimensional $k$-vector space $\bV$ and let $\GL$ be the group of $k$-linear automorphisms of $\bV$. A \emph{polynomial representation} of $\GL$ is one that is a subquotient of (arbitrary) direct sums of tensor powers of $\bV$; i.e., the representation is contained in the smallest symmetric monoidal abelian category generated by $\bV$ in the category of all representations of (the $k$-points of) $\GL$. The category of polynomial representations of $\GL$ will be denoted by $\Rep^{\pol}(\GL)$. 

A $\GL$-algebra is a commutative algebra object in the category $\Rep^{\pol}(\GL)$.  See \cite[Section~2]{gan22ext} for some preliminary definitions and results.

Given a $\GL$-algebra $R$, we let $\Mod_R$ be the category of module objects for $R$ in $\Rep^{\pol}(\GL)$. An $R$-module is induced if it is of the form $R \otimes W$ for some finite length polynomial representation $W$ and it is semi-induced if it has a finite filtration where the successive quotients are induced.

A $\GL$-algebra $R$ is a $\GL$-domain if for all nonzero $\GL$-stable ideals $\fa, \fb$, we have $\fa \fb \ne 0$. For a $\GL$-domain $R$, we let $\Mod_R^{\tors}$ denote the Serre subcategory of torsion $R$-modules, i.e., $R$-modules locally annihilated by a nonzero ideal. The generic category is $\Mod_R^{\gen} = \Mod_R/\Mod_R^{\tors}$ with the canonical quotient functor being $T$ which has right adjoint $S$.

\subsection{Modules over truncated polynomial rings}
Let $A = \Sym(\bV)$ be the infinite-variable polynomial ring and $\fm \subset A$ be the homogeneous maximal ideal. The ideal $\fm^{[q]}$ is the Frobenius power of $\fm$ generated by the $q$-th power of all linear forms in $A$. 

We now apply the framework from Section~\ref{s:axioms} to $\GL$-equivariant $\Sm$-modules. The $\GL$-algebra $\Sm$ is a $\GL$-domain; let  $\cL = \Mod_{\Sm}$
so $\cM = \Mod_{\Sm}^{\gen}$. 

\begin{theorem}\label{thms} \leavevmode
\begin{enumerate}
    \item The category $\cL$ is locally noetherian.
    \item Assume $M$ and $N$ are finitely generated $\Sm$-modules. For all $i \geq 0$, we have $\dim_k(\ext^i_{\cL}(M, N)) < \infty$.
    \item Assume $C$ is a finitely generated generic $\Sm$-module. There exists an injection $C \to I$ where  $I = T(N)$ is the image under $T$ of a finitely generated semi-induced $\Sm$-module $N$.
    \item Assume $M$ is a semi-induced $\Sm$-module. The map $M \to ST(M)$ is an isomorphism and $\rR^iS(T(M)) = 0$ for all $i > 0$.    
    \item Assume $C$ is a finitely generated object in $\cM$. The $\Sm$-module $\rR^iS (C)$ is finitely generated for all $i \geq 0$ and vanishes for $i \gg 0$.
\end{enumerate}
\end{theorem}
\begin{proof}
\begin{enumerate}
    \item See \cite[Theorem~3.1]{gan25pol} and references therein.
    \item Since $M$ is finitely generated and $\Mod_{\Sm}$ is locally noetherian, there exists a projective resolution 
   \[
   \ldots \to P_1 \to P_0 \to M \to 0
   \]
   with $P_i$ being a finitely generated projective $\Sm$-module for $i \geq 0$.
   Applying the functor $\Hom_{\Sm}(-, N)$ to the above, we obtain a complex
   \[
    0 \to \Hom_{\cL}(P_0, N) \to \Hom_{\cL}(P_1, N) \to \ldots
   \]
   whose homology is $\ext^i_{\cL}(M, N)$. Each term of the above complex is a finite-dimensional vector space over $k$, whence the result follows.
   \item Let $M$ be a finitely generated $\Sm$-module such that $T(N) \cong C$. By the shift theorem \cite[Theorem~E]{gan25pol}, we have a map $M \to N$ with torsion kernel, where $N$ is a finitely generated semi-induced module. Applying the functor $T$, we thus get an injection $T(M) \cong C \to I \coloneqq T(N)$ as required.
   \item This is \cite[Corollary~3.8]{gan24monomial}.
   \item This is \cite[Proposition~3.9]{gan24monomial}.
\end{enumerate}
\end{proof}

 For $n \geq 0$, let $\cL_n$ be the set of irreducible polynomial representations of $\GL$ of degree $n$. Recall that for each partition $\lambda \vdash n$, there is a unique irreducible polynomial representation of degree $n$ with highest weight $\lambda$ denoted $L_{\lambda}$.
 
For $\lambda \in \cL_n$, let $K_{\lambda} = T(\Sm \otimes L_{\lambda})$
and $I_{\lambda} = T(\Sm \otimes \Theta({\lambda}))$ where $\Theta(\lambda)$ is the injective envelope of $L_{\lambda}$ in $\Rep^{\pol}(\GL)$. It follows by Theorem~\ref{thms}(d) that an object in $\cM$ is $K$-filtered if and only if its saturation is semi-induced.

The verification of (A1)--(A11) is now entirely analogous to the previous section. 

\begin{lemma}\label{lem:gssB}
   Let $N$ be an object in $\cM$ such that $S(N)$ is semi-induced and $M$ be an arbitrary $\Sm$-module. We have isomorphisms $\ext^i_{\cL}(M, S(N)) \cong \ext^i_{\cM}(T(M), N)$ for all $i \geq 0$.
\end{lemma}
\begin{proof}
   The proof of Lemma~\ref{lem:gss} applies mutatis mutandis. 
\end{proof}

\begin{corollary}\label{cor:adjunctionB}
 Let $W$ be a degree $d$ polynomial representation of $\GL$ and $N$ be a semi-induced $\Sm$-module. For $i \geq 0$, we have isomorphisms
  \[\ext^i_{\cM}( T(\Sm \otimes W), T(N)) \cong \ext^i_{\GL}(W, N_d).\]  
\end{corollary}
\begin{proof}
   The proof of Corollary~\ref{cor:adjunction} applies mutatis mutandis. 
\end{proof}

\begin{corollary}\label{cor:extvanishingB}
    Let $\lambda, \mu \in \cL$ satisfy $|\lambda| < |\mu|$. For all $i\geq 0$, we have
    \[\ext^i_{\cM}(K_{\lambda}, K_{\mu}) = 0.\] 
\end{corollary}
\begin{proof}
   The proof of Corollary~\ref{cor:extvanishing} applies mutatis mutandis. 
\end{proof}

\begin{corollary}\label{cor:extvanishing2B}
    Let $\lambda, \mu \in \cL$ satisfy $|\lambda| \leq |\mu|$. For all $i > 0$, we have
    \[\ext^i_{\cM}(K_{\lambda}, I_{\mu}) = 0.\] 
\end{corollary}
\begin{proof}
   The proof of Corollary~\ref{cor:extvanishing2} applies mutatis mutandis. 
\end{proof}

\begin{lemma}\label{lem:ext1B}
    For finitely generated $C$ and $D$ in $\cM$ and $i \in \{0,1\}$,  we have 
    \[\dim_k(\ext^i_{\cM}(C, D)) < \infty.\]
\end{lemma}
\begin{proof}
   The proof of Lemma~\ref{lem:ext1} applies mutatis mutandis. 
\end{proof}

\begin{lemma}\label{lem:estimate} \leavevmode
    \begin{enumerate}
    \item For $\lambda \in \cL$, the graded object $S(K_{\lambda})(k^n)$ is nonzero in degree $(q-1)n + |\lambda|$ for $n \gg 0$ and is zero above that degree.
    \item Assume $C$ is a gr-subK object of degree $<d$ in $\cM$. The graded object $S(C)(k^n)$ is supported in degrees $< (q-1)n + d$ for $n \gg 0$.
    \end{enumerate}
\end{lemma}
\begin{proof}
   (a) Since $S(K_{\lambda}) \cong \Sm \otimes L_{\lambda}$ by Theorem~\ref{thms}(d), this result is \cite[Lemma~5.12]{gan25pol}.
   
   (b) This follows from the first part and left exactness of $S$.
\end{proof}

\begin{lemma}\label{lem:noinjectiveB}
    Assume $C$ is a gr-subK object of degree $<d$ in $\cM$. For $\lambda \in \cL_d$, there are no injective maps $K_{\lambda} \to C$.
\end{lemma}
\begin{proof}
    By applying $S$ to a map $\phi \colon K_{\lambda} \to C$, we obtain a map $S(\phi) \colon S(K_{\lambda}) \cong \Sm \otimes L_\lambda \to S(C)$. In particular, we also obtain a map $S(\phi)(k^n) \colon (\Sm \otimes L_{\lambda})(k^n) \to S(C)(k^n)$ for $n \gg 0$. By Lemma~\ref{lem:estimate}(a) and (b), the graded object $(\Sm \otimes L_{\lambda})(k^n)$ is nonzero in degree $(q-1)n+d$ whereas the graded object $S(C)(k^n)$ vanishes in that degree. Hence $S(\phi)(k^n)$ is not injective, and in turn, neither is $\phi$.
\end{proof}

\begin{lemma}\label{lem:nononinjB}
    Let $\phi\colon K_{\lambda} \to I_{\mu}$ be a nonzero map with $|\mu| = |\lambda|$. Then (1) $\mu = \lambda$ and (2) $\phi$ is injective.
\end{lemma}
\begin{proof}
    The proof of Lemma~\ref{lem:nononinj} applies mutatis mutandis.
\end{proof}

\begin{proposition}
    Properties (A1)--(A11) hold for $\cM$. 
\end{proposition}
\begin{proof}
\leavevmode
{\renewcommand{\theenumi}{A\arabic{enumi}}\renewcommand{\labelenumi}{\textup{(\theenumi)}}
\begin{enumerate}
\item clearly holds as the number of partitions of $n$ is finite.
\item is Lemma~\ref{lem:ext1B}.
\item is Corollary~\ref{cor:extvanishingB}. 
\item is Corollary~\ref{cor:extvanishing2B}.
\item holds since $\Theta(\mu)$ has a finite filtration where the successive quotients are irreducible $\GL$-representations of degree $|\mu|$ and the functors $T$ and $\Sm \otimes_k (-)$ are exact.
\item holds since $\End_{\cM}(I_{\lambda}) \cong  \End_{\GL}(\Theta(\lambda))$ by Corollary~\ref{cor:adjunctionB} which is a local ring since $\Theta(\lambda)$ is an indecomposable $\GL$-representation being the injective envelope of $L_{\lambda}$. 
\item is Lemma~\ref{lem:nononinjB}(2).
\item is Lemma~\ref{lem:noinjectiveB}.
\item is Lemma~\ref{lem:nononinjB}(1).
\item holds by Theorem~\ref{thms}(c) as the image of a semi-induced object in $\cM$ is $K$-filtered.
\item holds by Theorem~\ref{thms}(a).
\qedhere
\end{enumerate}}
\end{proof}

For $n \in \bN$, let $\cM_{\leq n}$ be the localizing subcategory generated by $T(\Sm \otimes L_{\lambda})$ with $|\lambda|\leq n$.  For $\lambda \in \cL_n$, we define $H_{\lambda}$ as the intersection of kernels of maps $K_{\lambda}$ to objects in $\cM_{\leq n-1}$. By Theorem~\ref{thm:main}, we have the following:
\begin{theorem}\label{thm:mainSm} 
For all $n \in \bN$, we have the following:
   \begin{enumerate}
        \item Every finitely generated object in $\cM_{\leq n}$ has finite length.
        \item The set $\{H_{\lambda}\}_{\lambda \in \cL_{\leq n}}$ is a complete set of pairwise non-isomorphic simples of $\cM_{\leq n}$.
        \item For $\lambda \in \cL_{\leq n}$, the object $K_{\lambda}/H_{\lambda}$ lies in $\cM_{\leq |\lambda|-1}$.
        \item For $\lambda \in \cL_{\leq n}$, the injective envelope $J_{\lambda, n}$ of $H_{\lambda}$ in $\cM_{\leq n}$ is of finite length.
        \item For $\lambda \in \cL_n$, we have $J_{\lambda, n} = I_{\lambda}$.
        \item For $\lambda \in \cL_{\leq n}$, the indecomposable injective $J_{\lambda, n}$ is gr-subK of degree $\leq n$.
    \end{enumerate}
\end{theorem}

\subsection{Modules over the exterior algebras}
The framework from Section~\ref{s:axioms} also appliies to $\GL$-equivariant modules over the infinite-variable exterior algebra $\lw(\bV)$; the necessary analogues of Theorem~\ref{thms} are proved in \cite{gan22ext}. We leave this easy extension to the reader.

\bibliographystyle{alpha} 
\bibliography{xxx}

\appendix
\section{Grothendieck abelian categories}\label{s:appx}
The results in this section are well-known; we include them for the sake of completeness. 
For the definition of a Grothendieck abelian category, see \cite[\href{https://stacks.math.columbia.edu/tag/079A}{Tag 079A}]{stacks-project}. Recall Grothendieck's criterion (AB5) is that all direct sums exist and filtered colimits are exact, or equivalently, all direct sums exist and filtered colimits commute with finite limits. We will use these two definitions interchangeably.

Throughout, we assume $\cA$ is a Grothendieck abelian category.

Given an object $N$ and a collection of subobjects $\{N_{\alpha} \hookrightarrow N \}_{\alpha \in J}$, we define their sum, denoted $\sum N_{\alpha}$ to be the smallest subobject of $N$ containing each $N_{\alpha}$. If the family is directed, then by (AB5) we have $\sum_{\alpha \in J} N_{\alpha} = \varinjlim_{\alpha \in J} N_{\alpha}$. 

An object $N$ is \emph{finitely generated} if for subobjects $\{N_i\}_{i \in J}$ satisfying $\sum_{i \in J} N_i = N$, there exists a finite subset $J' \subset J$ such that $\sum_{i \in J'} N_i = N$.

\begin{lemma}\label{lem:fg}
    Let $J$ be a directed set, and 
    $\{M_i\}_{i \in J}$ be a direct system in $\cA$.
    Let $M$ be an object in $\cA$ and $N \subset M$ be finitely generated. 
    Given a surjective map $\phi \colon \varinjlim M_i \to M$, there exists $j \in J$ such that $N \subset \phi(M_j)$.
\end{lemma}
\begin{proof}
   Replacing $M_i$ with its image inside $M$, we may assume we have subobjects $M_i \subset M$ and a directed sum $M = \sum_{i \in J} M_i$. 
   Let $N_i = N \cap M_i$ for all $i \in J$.  
   By (AB5), we have $N = \sum_{i \in J} (N \cap M_i)$. So finite generation of $N$ implies that there exists $j \in J$ such that $N = N \cap M_j$ or $N \subset M_j$, as required.
\end{proof}

\subsection{Baer's criterion}
We first prove Baer's criterion; the proof is a routine application of Zorn's lemma.
\begin{proposition}\label{prop:baer}
Assume every nonzero object in $\cA$ contains a nonzero finitely generated subobject. 
Let $I$ be an
object of $\cA$ such that $\ext^1_{\cA}(M, I) = 0$ for all finitely generated
objects $M$. Then $I$ is injective.
\end{proposition}
\begin{proof}
Let $N' \subset N$ be a subobject and $f \colon N' \to I$ be a map. To prove that $I$ is injective, it suffices to extend $N' \to I$ to a map $N \to I$. Our assumption is that this is possible whenever $N/N'$ is finitely generated.

Let 
\[\Sigma = \{(P, g) \,\vert\, N' \subset P \subset N, g \colon P \to I \text{ with } g|_{N'} = f\},\] 
with the partial ordering
$(P, g) \le (Q, h)$ if $P \subset Q$ and $h|_{P} = g$; it is clearly nonempty. 
Given a chain $\{P_{\alpha}, g_{\alpha}\}_{\alpha \in \Gamma}$ in $\Sigma$,
the sum $P \coloneqq \sum P_{\alpha} = \varinjlim P_{\alpha}$ is a subobject of $N$ by (AB5) containing $N'$ so $g \coloneqq \varinjlim g_{\alpha}$ is an extension of $f$ to $P$. So $(P, g)$ is an upper bound of the given chain in $\Sigma$. By Zorn's lemma, the set $\Sigma$ has a maximal element $(Q, h)$.

For the sake of contradiction, assume $Q \neq N$. 
By assumption, the object $N/Q$ contains a nonzero finitely generated subobject; let its preimage be $M \subset N$. By the first paragraph, the map $h \colon Q \to I$ extends to $h' \colon M \to I$ as $M/Q$ is finitely generated, contradicting maximality. So $Q = N$, as required.
\end{proof}

\subsection{Localizing subcategories}
A \emph{Serre subcategory} of $\cA$ is a full subcategory that is closed under subquotients and extensions. A Serre subcategory is a \emph{localizing subcategory} if it is closed under arbitrary direct sums, or equivalently, arbitrary direct limits. 
\begin{lemma}\label{lem:obvious}
    For objects $M$ and $N$ in a localizing subcategory $\cD \subset \cA$, we have $\ext^i_{\cD}(M, N) \cong \ext^i_{\cA}(M, N)$ for $i \in \{0,1\}$.
\end{lemma}
\begin{proof}
    For $i = 0$, this is the property that a localizing subcategory is full, and for $i=1$, this is the property that a localizing subcategory is closed under extensions.
\end{proof}

Assume $X$ is a collection of objects in $\cA$. The \emph{localizing subcategory generated by $X$} is the smallest localizing subcategory of $\cA$ containing $X$. This is well-defined because the arbitrary intersection of localizing subcategories is also a localizing subcategory.
\begin{definition}
An object $M$ is \emph{$X$-subfiltered} if there exists a finite filtration 
\[0 = M^0 \subset M^1 \subset \cdots \subset M^n = M\] 
such that for all $i \in [n]$, each $M^i/M^{i-1}$ is isomorphic to a subquotient of an object in $X$.
\end{definition}
Let $\cC_0$ be the full subcategory on the $X$-subfiltered objects. It is easy to check that $\cC_0$ is a Serre subcategory.
Let $\cC$ be the full subcategory on objects of $\cA$ all of whose finitely generated subobjects are in $\cC_0$. The goal of this section is to prove that the localizing subcategory generated by $X$ is $\cC$, provided the ambient category is locally noetherian.

\begin{lemma} \label{lem:locgen}
Assume $\cA$ is locally noetherian. The category $\cC$ is a localizing subcategory.
\end{lemma}
\begin{proof}
Let $M$ be an object in $\cC$, i.e., every finitely generated subobject of $M$ is in $\cC_0$. Closure of $\cC$ under subobjects is clear since a finitely generated subobject of a subobject of $M$ is a finitely generated subobject of $M$ thus in $\cC_0$.

Let $\pi \colon M \to N$ be a surjection with $M$ in $\cC$, and let $N_0 \subset N$ be finitely generated. Writing $\{M_i\}_{i \in I}$ for the (directed) family of finitely generated subobjects of
$M$, we get a surjection $\epsilon \colon \varinjlim M_i  \to M$. By Lemma~\ref{lem:fg}, there exists $j$ such that $N_0 \subset \pi(\epsilon(M_j))$, or $N_0$ is a subquotient of $M_j$. 
Since $M_j \subset M$ is finitely generated, it lies in $\cC_0$ and therefore, the object $N_0$ also lies in the Serre subcategory $\cC_0$. So every finitely generated subobject of $N$ lies in $\cC_0$, so by the definition of $\cC$, the object $N$ lies in $\cC$.

Next assume $0 \to N \to M \to Q \to 0$ is exact with $N$ and $Q$ in $\cC$, and let $M_0 \subset M$ be finitely generated. Then $N \cap M_0$ is finitely generated by noetherianity, and $M_0/(N \cap M_0) \cong (M_0 + N)/N$ is a finitely generated subobject of $Q$; both lie in $\cC_0$ by assumption so $M_0$ also lies in the Serre subcategory $\cC_0$.

Finally, assume $\{M_i\}_{i \in I}$ are objects of $\cC$ and $M = \bigoplus_{i \in I} M_i$. Let $N \subset M$ be a finitely generated subobject.   Setting $M_J = \bigoplus_{i \in J} M_i$ for all finite subsets $J \subset I$, we have $M = \varinjlim_{|J| < \infty} M_J$. So by Lemma~\ref{lem:fg}, we have $N \subset M_J$ for some finite set $J$. The object $M_J$ lies in $\cC$ by the previous paragraph, so the finitely generated subobject $N \subset M_J$ lies in $\cC_0$. 
\end{proof}

\begin{lemma}\label{lem:dirunion}
Assume $\cA$ is locally noetherian. Every object in $\cC$ is a directed union of objects in $\cC_0$.
\end{lemma}
\begin{proof}
   By local noetherianity, every object in $\cA$ is the directed union of its finitely generated subobjects. By definition of $\cC$, every finitely generated subobject of an object in $\cC$ lies in $\cC_0$. The result follows.
\end{proof}

\begin{proposition}\label{prop:AppA}
    Assume $\cA$ is locally noetherian. The localizing subcategory generated by $X$ is precisely $\cC$. In particular, every finitely generated object $M$ in the localizing subcategory generated by $X$ is $X$-subfiltered. 
\end{proposition}
\begin{proof}
    Let $E$ be an object in $X$. Every finitely generated subobject of $E$ is trivially $X$-subfiltered and thus lies in $\cC_0$. So by the definition of $\cC$, the object $E$ lies in $\cC$. Therefore, the localizing subcategory generated by $X$ is contained in $\cC$ by Lemma~\ref{lem:locgen}. 

    The category $\cC_0$ is contained in the localizing subcategory generated by $X$. By Lemma~\ref{lem:dirunion}, we thus get the reverse containment as localizing subcategories are closed under filtered colimits.
\end{proof}
\begin{example}\label{exm}
    We use the notations from this section in this example.
    Let $k$ be a field. Let $R = \Sym(k^{\infty})$ with homogeneous maximal ideal $\fm$ and $A = R/\fm^2$. 
    Consider $\cA = \Mod_A$ and $X = \{ A/\fm \cong k\}$. Let $\cD$ be the localizing subcategory generated by $X$. We claim $\cD = \Mod_A$: indeed, we see $A_1 = k^{\infty}$ lies in $\cD$ and so we see from the short exact sequence $0 \to A_1 \to A \to k \to 0$ that $A$ lies in $\cD$ as well. 
    The subcategory $\cC_0$ is the full subcategory on objects which are finite-dimensional over $k$.
    From this, we see that the subcategory $\cC$ does not contain $A$ as $A$ itself is finitely generated, so $\cC$ is strictly contained in $\cD$ without the noetherian hypothesis.
\end{example}
Despite the previous example, the localizing subcategory generated by simple objects will not contain other simples, as one expects, regardless of noetherian assumptions.
\begin{lemma}\label{lem:Appsimple}
   Let $X$ be a collection of simple objects in $\cA$. Assume $L$ is a simple object in the localizing subcategory generated by $X$. Then there exists $H \in X$ such that $L \cong H$. 
\end{lemma}
\begin{proof}
    Let $\cE$ be the full subcategory on objects $M$ such that for all injective objects $E$ in $\cA$ satisfying $\Hom(H, E)$ = 0 for all $H \in X$, we have $\Hom(M, E) = 0$ as well. We claim that $\cE$ is a Serre subcategory. Closure under quotients and extensions is clear. Now, assume $M' \subset M$ where $M$ lies in $\cE$, and $E$ is an injective such that $\Hom(H, E) = 0$ for all $H \in X$. Given a map $M' \to E$, by injectivity of $E$, this extends to a map $M \to E$ which must be zero as $M$ lies in $\cE$, so the original map $M' \to E$ is also zero, hence $M'$ also lies in $\cE$. The subcategory $\cE$ is also localizing as it is closed under direct sums: given a family of objects $\{M_i\}_{i \in J}$ with each $M_i$ lying in $\cE$ and an injective $E$ with $\Hom(H, E) = 0$ for all $H \in X$, we have $\Hom(\bigoplus_{i \in J} M_i, E) = \prod_{i \in J} \Hom(M_i, E) = \prod 0$. Thus, we see that the localizing subcategory generated by $X$ is contained in $\cE$.

    Given a simple object $L$ in $\cE$, let $E(L)$ be its injective envelope. Of course $\Hom(L, E(L)) \ne 0$. So by the description of $\cE$, there exists $H \in X$ such that $\Hom(H, E(L)) \ne 0$. However, the socle of $E(L)$ is simple whence we get $L \cong H$, as required.
\end{proof}



\end{document}